\documentclass[12pt,reqno]{amsart}

\usepackage[left=3.5cm,right=3.5cm]{geometry}
\usepackage{amssymb}
\usepackage{graphicx}
\usepackage{amscd}
\usepackage[hidelinks]{hyperref}
\usepackage{color}
\usepackage{tabularx}
\usepackage[table]{xcolor}
\usepackage{float}
\usepackage{graphics,amsmath,amssymb}
\usepackage{amsthm}
\usepackage{amsfonts}
\usepackage{latexsym}
\usepackage{epsf}
\usepackage{xifthen}
\usepackage{mathrsfs}
\usepackage{dsfont}
\usepackage{makecell}
\usepackage{subfig}
\usepackage{amsmath}
\allowdisplaybreaks[4]
\usepackage{listings}
\usepackage{etoolbox}
\usepackage{fancyhdr}
\usepackage{pdflscape}
\usepackage[title,toc,titletoc]{appendix}
\usepackage{enumitem}
\usepackage[noadjust]{cite}
\usepackage{tikz}
\usetikzlibrary{automata,positioning,arrows}
\usepackage{young}
\usepackage[object=vectorian]{pgfornament} %%  http://altermundus.com/pages/tkz/ornament/index.html
\usepackage{lipsum,tikz}
\usepackage{multirow}
\usepackage[OT2,T1]{fontenc}

\numberwithin{equation}{section}

\theoremstyle{theorem}
\newtheorem{theorem}{Theorem}[section]
\newtheorem*{theorem*}{Theorem}

\newtheorem{corollary}[theorem]{Corollary}

\newtheorem{proposition}[theorem]{Proposition}

\newtheorem{innercustomgeneric}{\customgenericname}
\providecommand{\customgenericname}{}
\newcommand{\newcustomtheorem}[2]{%
	\newenvironment{#1}[1]
	{%
		\renewcommand\customgenericname{#2}%
		\renewcommand\theinnercustomgeneric{##1}%
		\innercustomgeneric
	}
	{\endinnercustomgeneric}
}
\newcustomtheorem{ctheorem}{Theorem}
\newcustomtheorem{clemma}{Lemma}

\theoremstyle{definition}
\newtheorem{definition}[theorem]{Definition}

\newtheorem{example}[theorem]{Example}
\newtheorem*{example*}{Example}
\newtheorem*{examples*}{Examples}

\newtheorem*{remark*}{Remark}
\newtheorem*{remarks*}{Remarks}
\newtheorem*{note*}{Note}

\newtheoremstyle{named}{}{}{\itshape}{}{\bfseries}{.}{.5em}{#1\thmnote{ #3}}
\theoremstyle{named}

\DeclareSymbolFont{cyrletters}{OT2}{wncyr}{m}{n}
\DeclareMathSymbol{\Sh}{\mathalpha}{cyrletters}{"58}
\DeclareMathSymbol{\sh}{\mathalpha}{cyrletters}{"78}
\DeclareMathSymbol{\Shh}{\mathalpha}{cyrletters}{"57}
\DeclareMathSymbol{\shh}{\mathalpha}{cyrletters}{"77}
\newcommand{\patt}[2]{\langle #1_{\triangleright}, #2_{\triangleleft} \rangle}
\newcommand{\bpatt}[2]{\big\langle #1_{\triangleright}, #2_{\triangleleft} \big\rangle}

\newcommand{\Prob}{\mathbf{P}}

\newcommand{\bs}{\boldsymbol{s}}

\title[Euler's pentagonal number theorem]{A probabilistic proof of Euler's pentagonal number theorem}

\author[S. Chern]{Shane Chern}
\address{Fakult\"at f\"ur Mathematik, Universit\"at Wien, Oskar-Morgenstern-Platz 1, Wien 1090, Austria}
\email{chenxiaohang92@gmail.com, xiaohangc92@univie.ac.at}

\date{}



\begin{document}
	
\sloppy

\begin{abstract}
	We present a probabilistic proof of Euler's pentagonal number theorem based on a shuffling model.
\end{abstract}

\maketitle

\begin{flushright}
{\noindent\footnotesize\itshape To George Andrews, a great advisor, on his $\big\lfloor (\tanh 1;\tanh 1)_\infty^{-1} \big\rfloor$-th birthday.}
\end{flushright}

\bigskip\bigskip

\section{Prologue}

Euler's \emph{pentagonal number theorem} is recognized as one of the pioneering discoveries in the analysis of the infinite. Formally, it is the following series expansion.

\begin{theorem}\label{th:PNT}
	We have
	\begin{align}\label{eq:PNT}
		\prod_{j\ge 1} (1-q^j) = \sum_{n=-\infty}^\infty (-1)^n q^{\frac{n(3n+1)}{2}}.
	\end{align}
\end{theorem}

This result was first communicated from Euler to Daniel Bernoulli in a letter dated November 30, 1740\footnote{Euler's original letter, unfortunately, has not been preserved but the lines regarding the pentagonal number theorem were alluded to briefly in Bernoulli's reply dated January 28, 1741. See \cite[Brief 52]{Eul2016}.}; however, no proof was included therein. The seemingly earliest demonstration offered by Euler \cite{Eul1760} was published in 1760, which is a retelling of the idea in his letter to Goldbach \cite[Brief 144]{Eul2015} on June 9, 1750. Euler's proof, as briefly presented at the end of this paper, is inductive, relying on an ingenious arithmetic manipulation of the series.

Ever since its discovery, the pentagonal number theorem has shown broad and profound connections to other fields. As observed by Legendre \cite[pp.~132--133, \S{}458]{Leg1830}, the series expansion \eqref{eq:PNT} may be combinatorially interpreted in terms of integer partitions; this interpretation motivated Franklin's bijective proof \cite{Fra1881} half a century later. Also, Goncharova \cite{Gon1973} noticed that the pentagonal number theorem arises as a natural implication from the calculation of the cohomology of the negative part of the Witt algebra, and it is just a tip of a big picture that bonds Lie algebra, partition theory, and $q$-series. We refer the interested reader to a terse exposition by Andrews \cite{And1983}, and a more detailed summary on Euler's discovery by Bell \cite{Bel2010}.

The objective of the present paper is to interpret the pentagonal number theorem in a different way. To be specific, we propose a shuffling model in the next section and utilize it to provide a new probabilistic proof.

\section{A shuffling model}

Let $\mathbb{N}_+$ be the set of positive integers, and $\mathbb{N}:=\mathbb{N}_+ \cup \{0\}$ the set of natural numbers. Given $a,b \in \mathbb{N}$, we denote by $[a,b]$ the set of natural numbers $n$ such that $a\le n\le b$, and by $[a,\infty)$ the set of natural numbers $n$ such that $n \ge a$.

Throughout, we always assume that $q$ is a real number such that $0<q<1$.

In a classical \emph{probability model} $(\Omega,\Prob)$, the \emph{sample space} $\Omega$ describes the set of all possible outcomes, and the \emph{probability function} $\Prob:\Omega \to [0,1]$ assigns a probability with value taken from $[0,1]$ to each element in $\Omega$ such that
\begin{align*}
	\sum_{\omega\in \Omega}\Prob_{\omega} = 1.
\end{align*}
We may also attach an \emph{event space} $E$, which is a subset of $\Omega$, to the model in order to describe the outcomes of particular interest. By abuse of notation, we denote by $\Prob_E$ the total probability of all elements in the event space $E$. This \emph{probability space} is written as $(\Omega,E,\Prob)$.

Now fix a sequence $\mathfrak{s}_k:=(s_0,\ldots,s_{k-1},s_k)$ of $k+1$ different natural numbers for a certain $k\ge 0$. We define a probability model $\sh_{\mathfrak{s}_k}=(\Omega_{\mathfrak{s}_k},\Prob^{(\mathfrak{s}_k)})$ where the sample space is
\begin{align*}
	\Omega_{\mathfrak{s}_k}=\big\{\bs_0,\bs_1,\ldots,\bs_{k-1},\bs_k\big\},
\end{align*}
with\footnote{It is notable that throughout this paper, \emph{three} different fonts for ``$s$'' are used, including $s$, $\bs$ and $\mathfrak{s}$, and they should be carefully distinguished, especially for the latter two (in \emph{bold} and \emph{Fraktur} styles).}
\begin{align}\label{eq:sh-Omega}
	\bs_i = \bs_i(\mathfrak{s}_k) := \begin{cases}
		(s_0,\ldots,s_{k-1},s_k) & \text{if $i=0$},\\
		(s_0,\ldots s_{k-i-1},s_k,s_{k-i},\ldots,s_{k-1}) & \text{if $i\in [1,k-1]$},\\
		(s_k,s_0,\ldots,s_{k-1}) & \text{if $i=k$}.
	\end{cases}
\end{align}
The probability function is then assigned by
\begin{align}\label{eq:prob-sh-k}
	\Prob_{\bs_i}^{(\mathfrak{s}_k)} = \begin{cases}
		1-q & \text{if $i=0$},\\
		q^i(1-q) & \text{if $i\in [1,k-1]$},\\
		q^k & \text{if $i=k$}.
	\end{cases}
\end{align}

Note that the outcomes in the sample space are sequences permuted by the original one. Intuitively, we \emph{swap} $s_k$ with its preceding entry with a probability $q$ through a step-by-step procedure, and the program is terminated when a certain step of swapping is not executed or $s_k$ already reaches the leftmost position.

This procedure more or less aligns with \emph{card shuffling}. So although less accurate\footnote{The term \emph{shuffle} is usually reserved for permutations where one divides a sequence into a front and a tail part and then forms a permutation in which the order of the elements from the front part is preserved, and also the order of the elements from the tail part (thus imitating a riffle shuffle of cards). For example, in the \emph{shuffle algebra}, the multiplication is exactly defined using this shuffle.}, in the present paper, for $S$ a subset of $\mathbb{N}$, we say a sequence $\mathfrak{s}:=(s_0,s_1,s_2,\ldots)$ is a \emph{shuffle} (instead of using the usual term of \emph{permutation}) of $S$ if every number in $\mathfrak{s}$ belongs to $S$ and every number in $S$ appears exactly once in $\mathfrak{s}$.

We then describe our \emph{shuffling model} $\Sh$ according to a \emph{Markov process}:

\begin{definition}[Markov process for the shuffling model $\Sh$]\label{def:Markov}
	~
	
	\begin{itemize}[leftmargin=*,align=left,itemsep=3pt]
		\item {\bfseries\boldmath Time $0$:} We begin with the sequence $(0)$;
		
		\item {\bfseries\boldmath Time $k$ ($k\ge 1$):} For the sequence $\mathfrak{s}_{k-1}=(s_0,\ldots,s_{k-1})$ derived at Time $k-1$, we append the number $k$ to the right end to get $\hat{\mathfrak{s}}_{k}=(s_0,\ldots,s_{k-1},k)$ (so that the new sequence $\hat{\mathfrak{s}}_{k}$ is of length $k+1$) and then apply the probability model $\sh_{\hat{\mathfrak{s}}_{k}}$.
	\end{itemize}
\end{definition}

If we only look at the outcomes of this stochastic process at Time $k$ and denote the associated probability model as $\Sh_k:=(\Omega_k,\Prob^{(k)})$, it is clear that the sample space $\Omega_k$ covers all shuffles of $[0,k]$. Let $\mathfrak{s}_{k}=(s_0,\ldots s_{k-i-1},k,s_{k-i},\ldots,s_{k-1})$ be such a shuffle wherein $i\in [0,k]$. Then by removing $k$, the sequence $\mathfrak{s}_{k-1}=(s_0,\ldots,s_{k-1})$ is a shuffle of $[0,k-1]$ according to the construction of the Markov process, and more importantly, it is \emph{exactly} the one derived at Time $k-1$ in this Markov process, thereby lying in the sample space of $\Sh_{k-1}$. We also iteratively have the probability evaluation:
\begin{align*}
	\Prob_{\mathfrak{s}_{k}}^{(k)} = \Prob_{\mathfrak{s}_{k-1}}^{(k-1)}\cdot P_i,
\end{align*}
where $\Prob_{\mathfrak{s}_{k}}^{(k)}$ denotes the probability of the event $\mathfrak{s}_{k}$ with respect to the probability model $\Sh_{k}$ and
\begin{align*}
	P_i = \begin{cases}
		1-q & \text{if $i=0$},\\
		q^i(1-q) & \text{if $i\in [1,k-1]$},\\
		q^k & \text{if $i=k$},
	\end{cases}
\end{align*}
according to the probability model $\sh_{\hat{\mathfrak{s}}_k}$.

Now we read our \emph{shuffling model} $\Sh = (\Omega,\Prob)$ as $\Sh_{\infty}$. Then the sample space $\Omega$ of $\Sh$ covers all shuffles of $\mathbb{N}$. 

\begin{example}
	Consider the shuffle $(2,0,3,1,4,5,6,7\ldots)$ of $\mathbb{N}$. Then in the Markov process, at Time $0$ the resulting event is $(0)$; at Time $1$ the resulting event is $(0,1)$; at Time $2$ the resulting event is $(2,0,1)$; at Time $3$ the resulting event is $(2,0,3,1)$; at Time $4$ the resulting event is $(2,0,3,1,4)$; and so on.
\end{example}

\section{Event spaces}

In this section, we describe the event spaces of our interest attached to the shuffling model $\Sh=(\Omega,\Prob)$.

For $\mathfrak{s}_k=(s_0,\ldots,s_k)$ a generic sequence of different natural numbers of length $k+1$, recall that we have defined a probability model $\sh_{\mathfrak{s}_k}=(\Omega_{\mathfrak{s}_k},\Prob^{(\mathfrak{s}_k)})$ in the previous section, where we have in particular labeled elements in the sample space $\Omega_{\mathfrak{s}_k}$ by $\bs_0,\ldots,\bs_k$ as recorded in \eqref{eq:sh-Omega}. Now in this model, we further attach an event space $E_{\mathfrak{s}_k}\subset \Omega_{\mathfrak{s}_k}$ and call the probability space $\shh_k = \shh_k(\mathfrak{s}_k) := (\Omega_{\mathfrak{s}_k},E_{\mathfrak{s}_k},\Prob^{(\mathfrak{s}_k)})$ a \emph{shuffling (restricted to $E_{\mathfrak{s}_k}$)}. For any element in $\Omega_{\mathfrak{s}_k}$, we say it is \emph{produced} by this shuffling $\shh_k$ if it belongs to $E_{\mathfrak{s}_k}$; otherwise, it is \emph{not} produced by $\shh_k$.

We make the following categorization:
\begin{itemize}[leftmargin=*,align=left,itemsep=3pt]
	\item \textsc{Type} $\triangleright$: The shuffling $\shh_k$ is \emph{leftmost} if the event space $E_{\mathfrak{s}_k}$ is $\{\bs_k\}$;
	
	\item \textsc{Type} $\triangleleft$: The shuffling $\shh_k$ is \emph{anti-leftmost} if the event space $E_{\mathfrak{s}_k}$ is $\{\bs_0,\ldots,\bs_{k-1}\}$;
	
	\item \textsc{Type} $\scalebox{0.725}{$\lozenge$}$: The shuffling $\shh_k$ is \emph{free} if the event space $E_{\mathfrak{s}_k}$ is $\{\bs_0,\ldots,\bs_{k-1},\bs_k\}$.
\end{itemize}

Next, we move on to the shuffling model $\Sh$. Recall that $\Sh$ is defined based on a Markov process introduced in Definition~\ref{def:Markov} of the previous section. Given any outcome $\mathfrak{s}\in \Omega$ of the model $\Sh$, it is derived by the following infinite chain
\begin{align*}
	\mathfrak{s}_0 \to \mathfrak{s}_1 \to \cdots \to \mathfrak{s}_k \to \cdots,
\end{align*}
where, for the moment, $\mathfrak{s}_k$ is the shuffle obtained at Time $k$ of our Markov process. Now it is notable that $\mathfrak{s}_k$, called the \emph{projection} of $\mathfrak{s}$ onto the model $\Sh_k$, belongs to the sample space $\Omega_k$. In light of what has been done at Time $k$ of the Markov process, $\mathfrak{s}_k$ may also be viewed as an outcome of the model $\sh_{\hat{\mathfrak{s}}_k}$ where $\hat{\mathfrak{s}}_k = (s_0,\ldots,s_{k-1},k)$ with $(s_0,\ldots,s_{k-1}) =: \mathfrak{s}_{k-1}$.

The takeaway in our analysis is that once $\mathfrak{s}_{k-1}\in \Omega_{k-1}$ is determined at Time $k-1$ of the Markov process, we will not only focus on all possible outcomes of the above $\sh_{\hat{\mathfrak{s}}_k}$. Instead, we look at a shuffling $\shh_k = \shh_k(\hat{\mathfrak{s}}_k) = (\Omega_{\hat{\mathfrak{s}}_k},E_{\hat{\mathfrak{s}}_k},\Prob^{(\hat{\mathfrak{s}}_k)})$ restricted to a certain event space $E_{\hat{\mathfrak{s}}_k}\subset \Omega_{\hat{\mathfrak{s}}_k}$. More precisely, we require that $\shh_k$ belongs to one of the three types introduced earlier: ``leftmost'', ``anti-leftmost'', or ``free''. By doing so, our attention is indeed restricted to an event space $E_k \subset \Omega_k$ attached to the model $\Sh_k$.

\begin{example}
	Consider the Markov process where the outcomes at Time $0$, $1$, and $2$ are given by
	\begin{align*}
		(0) \to (0,1) \to (2,0,1) \to \cdots.
	\end{align*}
	Then at Time $3$,
	\begin{itemize}[leftmargin=*,align=left,itemsep=3pt]
		\item The outcome produced by a leftmost shuffling $\shh_3$ includes $(3,2,0,1)$;
		
		\item The outcomes produced by an anti-leftmost shuffling $\shh_3$ include $(2,3,0,1)$, $(2,0,3,1)$, and $(2,0,1,3)$;
		
		\item The outcomes produced by a free shuffling $\shh_3$ include $(3,2,0,1)$, $(2,3,0,1)$, $(2,0,3,1)$, and $(2,0,1,3)$.
	\end{itemize}
\end{example}

The same idea may be repeated at each time of the Markov process.

For the moment, let us mark each index in the \emph{time set} $\mathbb{N}_+$ with a chosen subscript:
\begin{align*}
	\mathbf{T} := \big\{k_{\square_k}:k\in \mathbb{N}_+\big\},
\end{align*}
where $\square_k \in \{\triangleright,\triangleleft,\scalebox{0.8}{$\lozenge$}\}$; such $\mathbf{T}$ is called a \emph{marked time set}.

Now starting with $E_0(\mathbf{T})=\big\{(0)\big\}=\Omega_0$, we may iteratively construct an event space $E_k(\mathbf{T})\subset \Omega_k$ attached to the model $\Sh_k$ by
\begin{align*}
	E_k(\mathbf{T}) := \bigcup_{\mathfrak{s}_{k-1} \in E_{k-1}(\mathbf{T})}\left\{\begin{gathered}
		\text{outcomes of the shuffling}\\
		\text{$\shh_k = (\Omega_{\hat{\mathfrak{s}}_k},E_{\hat{\mathfrak{s}}_k},\Prob^{(\hat{\mathfrak{s}}_k)})$ of Type $\square_k$}
	\end{gathered}\right\},
\end{align*}
where as before we write $\hat{\mathfrak{s}}_k = (s_0,\ldots,s_{k-1},k)$ with $(s_0,\ldots,s_{k-1}) =: \mathfrak{s}_{k-1}$. Eventually, we arrive at an event space attached to our shuffling model $\Sh = (\Omega,\Prob)$:
\begin{align*}
	E(\mathbf{T}) := E_\infty(\mathbf{T}) \subset \Omega.
\end{align*}

\begin{example}
	We fix the marked time set as $\{1_{\lozenge},2_{\triangleleft},3_{\triangleright},\ldots\}$. Since we witness a free ($\scalebox{0.8}{$\lozenge$}$) shuffling at Time $1$, then
	\begin{align*}
		E_1(\{1_{\lozenge},2_{\triangleleft},3_{\triangleright},\ldots\}) = \big\{(0,1), (1,0)\big\};
	\end{align*}
	next, since we witness an anti-leftmost ($\triangleleft$) shuffling at Time $2$, then
	\begin{align*}
		E_2(\{1_{\lozenge},2_{\triangleleft},3_{\triangleright},\ldots\}) = \big\{(0,1,2), (0,2,1), (1,0,2), (1,2,0)\big\};
	\end{align*}
	next, since we witness a leftmost ($\triangleright$) shuffling at Time $3$, then
	\begin{align*}
		E_3(\{1_{\lozenge},2_{\triangleleft},3_{\triangleright},\ldots\}) = \big\{(3,0,1,2), (3,0,2,1), (3,1,0,2), (3,1,2,0)\big\}.
	\end{align*}
\end{example}

If we define
\begin{align*}
	I&:=\big\{k\in \mathbb{N}_+: \square_k = \triangleright\big\},\\
	J&:=\big\{k\in \mathbb{N}_+: \square_k = \triangleleft\big\},
\end{align*}
it is clear that $I$ and $J$, as subsets of $\mathbb{N}_+$, are \emph{disjoint}. Conversely, suppose $I$ and $J$ are two disjoint subsets of $\mathbb{N}_+$. Then we may recover a marked time set $\{1_{\square_1}, 2_{\square_2}, 3_{\square_3}, \ldots\}$ according to the rule:
\begin{align*}
	\square_k = \begin{cases}
		\triangleright & \text{if $k\in I$},\\
		\triangleleft & \text{if $k\in J$},\\
		\scalebox{0.8}{$\lozenge$} & \text{if $k\not\in I$ and $k\not\in J$}.
	\end{cases}
\end{align*}
Accordingly, we have an event space $E(\{1_{\square_1}, 2_{\square_2}, 3_{\square_3}, \ldots\})$ attached to the model $\Sh$, and we abbreviate it by the notation:
\begin{align*}
	\bpatt{I}{J} := E(\{1_{\square_1}, 2_{\square_2}, 3_{\square_3}, \ldots\}).
\end{align*}

\begin{example}
	Consider the shuffle
	\begin{align*}
		(6, 1, 0, 4, 2, 3, 5, 7, 8, 9, 10, \ldots) \in \Omega.
	\end{align*}
	It belongs to the event spaces $\bpatt{\varnothing}{\varnothing}$, $\bpatt{\{1\}}{\varnothing}$, $\bpatt{\{1,6\}}{\varnothing}$, $\bpatt{\varnothing}{\{2\}}$, $\bpatt{\varnothing}{\{4,5\}}$, $\bpatt{\{1\}}{\{2,4\}}$ (and many others) but not the space $\bpatt{\{2\}}{\varnothing}$ (because the projection of shuffles in this space onto $\Sh_2$ should be either $(2,0,1)$ or $(2,1,0)$) nor $\bpatt{\varnothing}{\{1\}}$ (because the projection of shuffles in this space onto $\Sh_1$ should be $(0,1)$).
\end{example}

We close this section with two observations.

\begin{proposition}\label{prop:disjoint}
	Let $I$ and $J$ be two disjoint subsets of $\mathbb{N}_+$, and $I'$ and $J'$ another two disjoint subsets of $\mathbb{N}_+$. If there exists a positive integer $k$ such that $k\in I$ and $k\in J'$, then
	\begin{align}
		\bpatt{I}{J} \cap \bpatt{I'}{J'} = \varnothing.
	\end{align}
\end{proposition}

\begin{proof}
	Assume that there exists a certain shuffle $\mathfrak{s}$ in both $\bpatt{I}{J}$ and $\bpatt{I'}{J'}$. We project $\mathfrak{s}$ onto the model $\Sh_k$ to get a shuffle $\mathfrak{s}_k\in \Omega_k$. Since $\mathfrak{s} \in \bpatt{I}{J}$ and $k\in I$, the shuffle $\mathfrak{s}_k$ is produced by a leftmost shuffling at Time $k$. Meanwhile, we see from the assumptions $\mathfrak{s} \in \bpatt{I'}{J'}$ and $k\in J'$ that $\mathfrak{s}_k$ is also produced by an anti-leftmost shuffling. However, the two properties cannot occur simultaneously, and hence we are led to a contradiction.
\end{proof}

Next, we recall that the union of two sets $S_1$ and $S_2$ is written as $S_1 \sqcup S_2$ once $S_1$ and $S_2$ are disjoint; this union is called a \emph{disjoint union}.

\begin{proposition}\label{prop:IJK-event}
	Let $I$ and $J$ be disjoint subsets of $\mathbb{N}_+$. Then for any subset $K$ of $\mathbb{N}_+$ that is disjoint from both $I$ and $J$,
	\begin{align}
		\bpatt{I}{J} = \bigsqcup_{K'\subset K}\bpatt{(I\sqcup K')}{(J\sqcup (K\backslash K'))}.
	\end{align}
	In particular, if $K$ is a singleton, namely, $K=\{k\}$ for a certain $k\in \mathbb{N}_+$, then
	\begin{align}
		\bpatt{I}{J} = \bpatt{(I\sqcup \{k\})}{J} \sqcup \bpatt{I}{(J\sqcup \{k\})}.
	\end{align}
\end{proposition}

\begin{proof}
	We only need the fact that at an arbitrary Time $k\in K$, the event space of a free shuffling $\shh_k$ is the disjoint union of the event spaces of a leftmost $\shh_k$ and an anti-leftmost $\shh_k$.
\end{proof}

\section{Probabilities}

In the previous section, we have determined the desired event spaces $\bpatt{I}{J}$ attached to the shuffling model $\Sh=(\Omega,\Prob)$ where $I,J\subset \mathbb{N}_+$ are disjoint. Now our task is to analyze the probability $\Prob_{\patt{I}{J}}$.

Recall that the model $\Sh$ is defined according to a Markov process. In particular, at Time $k$, we focus on a shuffling $\shh_k = \shh_k(\hat{\mathfrak{s}}_k)$, which is a probability space $(\Omega_{\hat{\mathfrak{s}}_k},E_{\hat{\mathfrak{s}}_k},\Prob^{(\hat{\mathfrak{s}}_k)})$ revolving around a certain event space $E_{\hat{\mathfrak{s}}_k}$, where $\mathfrak{s}_{k-1}=(s_0,\ldots,s_{k-1})$ is the outcome at Time $k-1$ and $\hat{\mathfrak{s}}_k = (s_0,\ldots,s_{k-1},k)$. Note that if $\shh_k$ is of one of the three types categorized in the previous section, the event space $E_{\hat{\mathfrak{s}}_k}$ is also decided accordingly. In light of the probability function \eqref{eq:prob-sh-k}, we compute that
\begin{align*}
	\Prob_{E_{\hat{\mathfrak{s}}_k}}^{(\hat{\mathfrak{s}}_k)} = \begin{cases}
		q^k & \text{if $\shh_k$ is leftmost},\\
		1-q^k & \text{if $\shh_k$ is anti-leftmost},\\
		1 & \text{if $\shh_k$ is free}.
	\end{cases}
\end{align*}
In particular, this probability evaluation is \emph{independent} of $\hat{\mathfrak{s}}_k$, or equivalently, the outcome $\mathfrak{s}_{k-1}$ at Time $k-1$.

This independence, together with the nature of a Markov process, immediately gives the following probability evaluation for the event space $\bpatt{I}{J}$ attached to the shuffling model $\Sh$.

\begin{proposition}\label{prop:prob-expression}
	Let $I$ and $J$ be disjoint subsets of $\mathbb{N}_+$. Then
	\begin{align}
		\Prob_{\patt{I}{J}} = \prod_{i\in I} q^i \prod_{j\in J} (1-q^j).
	\end{align}
\end{proposition}

Proposition \ref{prop:prob-expression} has two implications.

\begin{corollary}\label{coro:I-infinite}
	Let $I$ and $J$ be two disjoint subsets of $\mathbb{N}_+$ with $I$ being infinite. Then
	\begin{align}
		\Prob_{\patt{I}{J}} = 0.
	\end{align}
\end{corollary}

\begin{corollary}\label{coro:I=I'}
	Let $I$ and $I'$ be two finite subsets of $\mathbb{N}_+$, not necessarily disjoint, and let $J$ be a subset of $\mathbb{N}_+$ that is disjoint from both $I$ and $I'$. Suppose that
	\begin{align*}
		\sum_{i\in I} i = \sum_{i'\in I'} i'.
	\end{align*}
	Then
	\begin{align}
		\Prob_{\patt{I}{J}} = \Prob_{\patt{I'}{J}}.
	\end{align}
\end{corollary}

Finally, we recall that, in a probability model, if two event spaces $E_1$ and $E_2$ are disjoint, then
\begin{align*}
	\Prob_{E_1\sqcup E_2} = \Prob_{E_1} + \Prob_{E_2}.
\end{align*}
Therefore, Proposition \ref{prop:IJK-event} implies the following relation.

\begin{proposition}\label{prop:IJK-prob}
	Let $I$ and $J$ be disjoint subsets of $\mathbb{N}_+$. Then for any subset $K$ of $\mathbb{N}_+$ that is disjoint from both $I$ and $J$,
	\begin{align}
		\Prob_{\patt{I}{J}} = \sum_{K'\subset K}\Prob_{\patt{(I\sqcup K')}{(J\sqcup (K\backslash K'))}}.
	\end{align}
	In particular, if $K$ is a singleton, namely, $K=\{k\}$ for a certain $k\in \mathbb{N}_+$, then
	\begin{align}
		\Prob_{\patt{I}{J}} = \Prob_{\patt{(I\sqcup \{k\})}{J}} + \Prob_{\patt{I}{(J\sqcup \{k\})}}.
	\end{align}
\end{proposition}

\section{The pentagonal number theorem}

The key to our probabilistic proof of the pentagonal number theorem is the following relation.

\begin{theorem}
	For every nonnegative integer $N$,
	\begin{align}\label{eq:prob-induction}
		\Prob_{\patt{\varnothing}{[1,\infty)}} &= \sum_{n=0}^N (-1)^n \big(\Prob_{\patt{[n+1,2n]}{\varnothing}} - \Prob_{\patt{[n+1,2n+1]}{\varnothing}}\big)\notag\\
		&\qquad\qquad\qquad\qquad - (-1)^{N} \sum_{k\ge 1} \Prob_{\patt{[k+1+N,k+1+2N]}{[1+N,k+N]}}.
	\end{align}
\end{theorem}

We show this result by induction on $N$.

\begin{proof}[Base case]
	We need to prove the case where $N=0$:
	\begin{align*}
		\Prob_{\patt{\varnothing}{[1,\infty)}} &= \Prob_{\patt{\varnothing}{\varnothing}} - \Prob_{\patt{\{1\}}{\varnothing}} - \sum_{k\ge 1} \Prob_{\patt{\{k+1\}}{[1,k]}}\\
		&= \Prob_{\patt{\varnothing}{\varnothing}} - \sum_{k\ge 1} \Prob_{\patt{\{k\}}{[1,k-1]}}.
	\end{align*}
	Note that $\bpatt{\varnothing}{\varnothing}$ contains all shuffles of $\mathbb{N}$ so $\bpatt{\varnothing}{\varnothing} = \Omega$. Meanwhile, given an arbitrary shuffle $\mathfrak{s}\in \Omega$, we project it onto each model $\Sh_k$ to get an outcome shuffle $\mathfrak{s}_k \in \Omega_k$ at Time $k$. Then either every $\mathfrak{s}_k$ is produced by an anti-leftmost shuffling (this case corresponds to $\bpatt{\varnothing}{[1,\infty)}$) or there exists a unique time $k$ such that $\mathfrak{s}_k$ is the first shuffle produced by a leftmost shuffling (this case corresponds to $\bpatt{\{k\}}{[1,k-1]}$). Furthermore, given two different positive integers $k$ and $k'$, we assume without loss of generality that $k<k'$. Then $k\in \{k\}$ and $k\in [1,k'-1]$. In light of Proposition \ref{prop:disjoint},
	\begin{align*}
		\bpatt{\{k\}}{[1,k-1]} \cap \bpatt{\{k'\}}{[1,k'-1]} = \varnothing.
	\end{align*}
	Hence,
	\begin{align*}
		\bpatt{\varnothing}{\varnothing} = \bpatt{\varnothing}{[1,\infty)} \sqcup \left(\bigsqcup_{k\ge 1} \bpatt{\{k\}}{[1,k-1]}\right).
	\end{align*}
	Application of Proposition \ref{prop:IJK-prob} leads us to the required relation.
\end{proof}

\begin{proof}[Inductive step]
	It suffices to show that, for each nonnegative integer $N$,
	\begin{align*}
		\sum_{k\ge 1} \Prob_{\patt{[k+1+N,k+1+2N]}{[1+N,k+N]}} &= \Prob_{\patt{[2+N,2+2N]}{\varnothing}} - \Prob_{\patt{[2+N,3+2N]}{\varnothing}}\\
		&\quad - \sum_{k\ge 1} \Prob_{\patt{[k+2+N,k+3+2N]}{[2+N,k+1+N]}}.
	\end{align*}
	We start by noting that, for each $k\ge 1$, the sets $[k+1+N,k+1+2N]$, $[2+N,k+N]$ and $\{1+N\}$ are pairwise disjoint. It follows from Proposition \ref{prop:IJK-prob} that
	\begin{align*}
		\sum_{k\ge 1} \Prob_{\patt{[k+1+N,k+1+2N]}{[1+N,k+N]}} &= \sum_{k\ge 1} \Prob_{\patt{[k+1+N,k+1+2N]}{[2+N,k+N]}}\\
		&\quad - \sum_{k\ge 1} \Prob_{\patt{(\{1+N\}\sqcup[k+1+N,k+1+2N])}{[2+N,k+N]}}.
	\end{align*}
	We split the first sum on the right-hand side as
	\begin{align*}
		\sum_{k\ge 1} \Prob_{\patt{[k+1+N,k+1+2N]}{[2+N,k+N]}} &= \Prob_{\patt{[2+N,2+2N]}{\varnothing}}\\
		&\quad + \sum_{k\ge 1} \Prob_{\patt{[k+2+N,k+2+2N]}{[2+N,k+1+N]}}.
	\end{align*}
	We further note that the sets $[k+2+N,k+2+2N]$, $[2+N,k+N]$ and $\{k+1+N\}$ are pairwise disjoint. Another application of Proposition \ref{prop:IJK-prob} yields
	\begin{align*}
		\sum_{k\ge 1} \Prob_{\patt{[k+2+N,k+2+2N]}{[2+N,k+1+N]}} &= \sum_{k\ge 1} \Prob_{\patt{[k+2+N,k+2+2N]}{[2+N,k+N]}}\\
		&\quad - \sum_{k\ge 1} \Prob_{\patt{[k+1+N,k+2+2N]}{[2+N,k+N]}}.
	\end{align*}
	Finally, we see from Corollary \ref{coro:I=I'} that
	\begin{align*}
		\sum_{k\ge 1} \Prob_{\patt{(\{1+N\}\sqcup[k+1+N,k+1+2N])}{[2+N,k+N]}} = \sum_{k\ge 1} \Prob_{\patt{[k+2+N,k+2+2N]}{[2+N,k+N]}}.
	\end{align*}
	This implies that
	\begin{align*}
		\sum_{k\ge 1} \Prob_{\patt{[k+1+N,k+1+2N]}{[1+N,k+N]}} &= \Prob_{\patt{[2+N,2+2N]}{\varnothing}}\\
		&\quad - \sum_{k\ge 1} \Prob_{\patt{[k+1+N,k+2+2N]}{[2+N,k+N]}}.
	\end{align*}
	By singling out the $k=1$ term of the sum in the above expression, we obtain the desired part $\Prob_{\patt{[2+N,3+2N]}{\varnothing}}$.
\end{proof}

Now we are ready to conclude the pentagonal number theorem.

\begin{proof}[Proof of Theorem \ref{th:PNT}]
	Letting $N\to \infty$ in \eqref{eq:prob-induction} and recalling Corollary \ref{coro:I-infinite}, we have
	\begin{align*}
		\Prob_{\patt{\varnothing}{[1,\infty)}} = \sum_{n\ge 0} (-1)^n \big(\Prob_{\patt{[n+1,2n]}{\varnothing}} - \Prob_{\patt{[n+1,2n+1]}{\varnothing}}\big).
	\end{align*}
	Application of Proposition \ref{prop:prob-expression} yields
	\begin{align*}
		\prod_{j\ge 1} (1-q^j) = \sum_{n\ge 0} (-1)^n \big(q^{\frac{n(3n+1)}{2}} - q^{\frac{(n+1)(3n+2)}{2}}\big).
	\end{align*}
	This is exactly \eqref{eq:PNT}.
\end{proof}

\section{Epilogue}

Before putting down the curtain, let us flash back to Euler's original proof of the pentagonal number theorem in \cite{Eul1760}. Adopting the \emph{$q$-Pochhammer symbol} for $n\in\mathbb{N}\cup\{\infty\}$,
\begin{align*}
	(A;q)_n&:=\prod_{j=0}^{n-1} (1-Aq^{j}),
\end{align*}
the key observation of Euler, although not explicitly stated, is the following.

\begin{theorem}
	For every nonnegative integer $N$,
	\begin{align}\label{eq:Euler-induction}
		(q;q)_\infty &= \sum_{n=0}^N (-1)^n \big(q^{\frac{n(3n+1)}{2}} - q^{\frac{(n+1)(3n+2)}{2}}\big)\notag\\
		&\quad - (-1)^{N} \sum_{k\ge 1} q^{\frac{(1+N)(2k+2+3N)}{2}} (q^{1+N};q)_k.
	\end{align}
\end{theorem}

It is worth noting that our proof can be understood as a translation of Euler's proof into the language of classical probability theory. To tie the knot, we simply apply Proposition \ref{prop:prob-expression} and derive that
\begin{align*}
	\Prob_{\patt{[k+1+N,k+1+2N]}{[1+N,k+N]}} &= \prod_{i=k+1+N}^{k+1+2N}q^i \prod_{j=1+N}^{k+N} (1-q^j)\\
	& = q^{\frac{(1+N)(2k+2+3N)}{2}} (q^{1+N};q)_k.
\end{align*}
In view of \eqref{eq:prob-induction}, this is clearly the last piece of expression required for \eqref{eq:Euler-induction}.

\subsection*{Acknowledgements}

This work was supported by the Austrian Science Fund (No.~10.55776/F1002). I would like to thank the anonymous referee for the inspiring suggestions that have improved the quality of this paper to a great extent. I am also grateful to Hugh Thomas for the comments on Proposition \ref{prop:IJK-prob} in an earlier draft.

\bibliographystyle{amsplain}

\end{document}